\documentclass[12pt,verbatim]{amsart}
\usepackage{amsmath,amssymb,amsfonts,amsthm}
\usepackage{dsfont}
\usepackage{times}
\usepackage{bm}
\usepackage{float}
\usepackage{indentfirst}
\usepackage{graphicx}
\usepackage{cite}

\usepackage{caption}
\usepackage{amsthm}
\usepackage{subfigure}
\usepackage{verbatim}
\usepackage{amscd}
\usepackage{ifthen}

\usepackage{color,xcolor}

\newtheorem{theorem}{Theorem}
\newtheorem{remark}{Remark}
\newtheorem{lemma}{Lemma}
\newtheorem{corollary}{Corollary}
\newtheorem{openproblem}{Open Problem}

\numberwithin{equation}{section}

\title{Monotonicity, convexity, and inequalities for functions involving gamma function}  
\author{Peipei Du} 
\address{Peipei Du, School of Science, Zhejiang Sci-Tech University, Hangzhou 310018, China}
\email{peipei\_du@126.com}

\author{Gendi Wang*} 
\address{*Gendi Wang ( Corresponding author ): School of Science, Zhejiang Sci-Tech University, Hangzhou 310018, China}
\email{gendi.wang@zstu.edu.cn}

\begin{document}  

\newcounter{minutes}\setcounter{minutes}{\time}
\divide\time by 60
\newcounter{hours}\setcounter{hours}{\time}
\multiply\time by 60 \addtocounter{minutes}{-\time}
\def\thefootnote{}
\footnotetext{ {\tiny File:~\jobname.tex,
          printed: \number\year-\number\month-\number\day,
          \thehours.\ifnum\theminutes<10{0}\fi\theminutes }}
\makeatletter\def\thefootnote{\@arabic\c@footnote}\makeatother

\maketitle

\begin{abstract}
In this paper,
we study some properties such as the monotonicity, logarithmically complete monotonicity,
logarithmic convexity, and geometric convexity, of the combinations of gamma function and power function.
The results we obtain generalize some related known results for parameters with specific values.
\end{abstract}



\section{Introduction}
The {\it gamma function} defined by
\begin{align*}
\Gamma(x)
=\int_0^\infty t^{x-1}e^{-t}{\rm d}t \quad ({\rm Re} \, x>0)
\end{align*}
is one of the most important functions in analysis and its application.

\medskip

The {\it psi} ({\it digamma}) {\it function}, the logarithmic derivative of the gamma function,
and the {\it polygamma functions} can be expressed as
\begin{equation*}
\psi(x)
=\frac{\Gamma'(x)}{\Gamma(x)}
=-\gamma+\int_0^\infty \frac{e^{-t}-e^{-xt}}{1-e^{-t}}{\rm d}t
\end{equation*}
and
\begin{equation*}
\psi^{(n)}(x)
=(-1)^{n+1}n! \sum_{k=0}^{\infty} \frac{1}{(x+k)^{n+1}}
\end{equation*}
for ${\rm Re}\,x>0\,, \, n=1,2,\cdots$\,, where
$\gamma=\lim\limits_{n\rightarrow\infty}\left(\sum\limits_{k=1}^n\dfrac1k-\log n\right)=0.57721\cdots$
is the Euler-Mascheroni constant.

\medskip

Let $I\subset(0\,,\infty)$ be an interval and $f:I\rightarrow(0\,,\infty)$ be a continuous function.
We say that $f$ is {\it geometrically convex (geometrically concave)}  on $I$ if the following is true:
\begin{equation*}
f(\sqrt{x_1x_2})
\leq(\geq)
\sqrt{f(x_1)f(x_2)}
\end{equation*}
for all $x_1,x_2\in I$\,, see \cite{PML,CPN}.

\medskip

Let $I\subset\mathbb{R}$ be an interval and $f:I\rightarrow(0\,,\infty)$ be a continuous function.
We say that $f$ is {\it logarithmically convex (logarithmically concave)}, log-convex (log-concave) for abbreviation, if
\begin{align*}
f\left(\frac{x+y}{2}\right)\leq(\geq)\sqrt{f(x)f(y)}
\end{align*}
for all $x\,,y\in I$\,, see \cite{CPN}.

\medskip

A function $f$ is called to be {\it logarithmically completely monotonic} ( LCM ) on an interval $I\subset\mathbb{R}$
if its logarithm $\log f$ satisfies
\begin{align}\label{ineq2}
(-1)^n\left(\log f(x)\right)^{(n)} \geq0
\end{align}
for all $x\in I$ and $n=1,2,\cdots$\,.
Moreover, the function $f$ is said to be strictly LCM on $I$ if the inequality (\ref{ineq2}) is strict, see \cite{ATR,FP}.
Clearly, the function $f$ is decreasing and log-convex if $f$ is LCM on $I$\,.
The analytical properties for the gamma function and related special functions have been extensively studied recently,
see \cite{GAI,YSZ,ZZH,QLD,QNL,BCK,TZH}.

\medskip

For $a>0\,,c\in\mathbb{R}$\,, let $x\in(-a\,,\infty) \backslash \{0\}$ and
\begin{align*}
f_{a,c,\pm1}(x)
\equiv \left(\frac{(\Gamma(x+a))^{\frac 1x}}{x^c}\right)^{\pm1}\,.
\end{align*}

The functions $f_{1,0,-1}$ and $f_{1,1,+1}$ are decreasing on $(0\,,\infty)$\,.
In addition, it is also proved that $f_{1,1-\gamma,-1}$ is decreasing on $(1\,,\infty)$, see \cite{DA}\,.
As a further study,
the function $f_{1,0,-1}$ is LCM on $(-1\,,\infty)$ \cite[Theorem 1]{BNF}
and $f_{1,0,+1}$ is geometrically convex on $(0\,,\infty)$ \cite[Theorem 1]{CZZ}\,.
Moreover, the conditions for $f_{1,c,\pm1}$ to be LCM on $(0\,,\infty)$ are shown \cite[Theorem 3, Theorem 4]{BNF}.
It is showed that
$f_{1,c,-1}$ $(f_{1,c,+1})$ is strictly decreasing on $(0\,,\infty)$
if and only if $c\leq0$ $(c\geq1)$ \cite [Theorem 4(b)]{SLCY}.

\medskip

For $a>0\,,c\in\mathbb{R}$\,, let $x\in(-a\,,\infty) \backslash \{0\}$ and
\begin{align*}
g_{a,c,\pm1}(x)
\equiv \left(\frac{(\Gamma(x+a))^{\frac 1x}}{(x+a)^c}\right)^{\pm1}\,.
\end{align*}

In \cite{FCP}, Theorem 1 shows that
$g_{1,1,+1}$ is strictly decreasing and strictly log-convex on $(0\,,\infty)$\,,
and Theorem 2 shows that $g_{1,\frac12,-1}$ is strictly decreasing and strictly log-convex on $(0\,,\infty)$.
As a generalization,
the conditions for $g_{1,c,\pm1}$ to be LCM on $(-1\,,\infty)$ are found \cite[Theorem 2]{BNF}.
Theorem 2 in \cite{SLCY} displays that
$g_{1,c,+1}$ $(g_{1,c,-1})$ is strictly decreasing on $(0\,,\infty)$
if and only if $c\geq1$ $\left(c\leq\frac{\pi^2}{12}\right)$\,,
and $g_{1,c,+1}$ ($g_{1,c,-1}$) is log-convex on $(0\,,\infty)$
if and only if $c\geq1$ ($c\leq c_0$)\,, where
$0.77797\cdots=\frac{75(28\zeta(3)+\pi^3)}{64}-75\leq c_0<18\left(3-\gamma-\log\pi-\frac{\pi^2}{8}\right)=0.79837\cdots$\,.
In addition, the conditions for $g_{a,c,\pm1}$ to be LCM on $(0\,,\infty)$
have been discussed \cite[Theorem 1.2, Remark 2.1]{MQA}\,.

\medskip

The purpose of the present paper is to further study the analytical properties of the gamma function.
Specifically, motivated by \cite[Theorem 1]{BNF} which says $f_{1,0,-1}$ is LCM on $(-1\,,\infty)$\,,
we study the monotonicity property of $f_{a,0,-1}$ for all $a\in\mathbb{R}$ and
find the necessary and sufficient conditions for $f_{a,0,-1}$ to be LCM
either on $(0\,,\infty)$ or on $(-a\,,\infty)$ in Theorem \ref{thm1}.
In Theorem \ref{thm2}, we also show the monotonicity, logarithmic convexity and geometric convexity properties
of $f_{a\,,c\,,+1}$ for certain values $(a,c)$\,,
which is a generalization about the specific parameters in some corresponding results in \cite{CZZ,BNF,SLCY}.
Similarly, in Theorem \ref{thm3} we investigate the monotonicity and geometric convexity properties
of $g_{a\,,c\,,+1}$ for certain values $(a,c)$\,
and obtain a generalization of some results in \cite{BNF,SLCY} and an improvement of the result in \cite{MQA}.

\medskip

Before presenting the main results,
we give some ranges of parameters, which are needed in describing the corollaries, as follows.

\medskip

Let
\begin{align*}
&D_1=\left\{(a,c)\Big|\frac12\leq a \leq 1\,, c\in \mathbb{R}\right\}\,,
& &D_2=\{(a,c)|a\geq2\,, c\in \mathbb{R}\}\,,\\
&D_3=\{(a,c)|1<a<2\,, c\geq0\}\,,
& &D_4=\{(a,c)|1<a<2\,, c\leq0\}\,,\\
&D_5=\left\{(a,c)\Big|\frac12\leq a \leq 1\,, c\geq1\right\}\,, \quad
& &D_6=\{(a,c)|a\geq2\,, c\geq1\}\,,\\
&D_7=\{(a,c)|a=1\,, c\leq0\}\,, \quad
& &D_8=\{(a,c)|a=2\,, c\leq0\}\,,\\
&D_9=\left\{(a,c)\Big|\frac12\leq a \leq 1\,, c\leq0\right\}\,, \quad
& &D_{10}=\{(a,c)|a\geq2\,, c\leq0\}\,,\\
&D_{11}=\left\{(a,c)\Big|a=2\,, c\leq\frac{\pi^2}{6}-1\right\}\,.
\end{align*}

\medskip

Let $a\in\mathbb{R}$\,, we define the function
\begin{align*}
g_1(x) \equiv f_{a,0,-1}(x)=\frac{1}{(\Gamma(x+a))^{\frac 1x}}\,,
\end{align*}
where $x\in(-a,\infty)$ for $a\leq0$\,; $x\in(-a,\infty) \backslash \{0\}$ for $a>0$\,.

Since
\begin{equation*}
g_1(0^-)=
\left\{
\begin{aligned}
&\infty\,, \quad &&0<a<1 \,\,\hbox{or}\,\, a>2\,,\\
&0\,, \quad      &&1<a<2\,,
\end{aligned}
\right.
\quad\quad
g_1(0^+)=
\left\{
\begin{aligned}
&0\,, \quad &&0<a<1 \,\,\hbox{or}\,\, a>2\,,\\
&\infty\,, \quad      &&1<a<2\,,
\end{aligned}
\right.
\end{equation*}
we only define $g_1(0)$ as follows
\begin{equation*}
g_1(0)=
\left\{
\begin{aligned}
&e^{\gamma}\,, \quad && a=1\,,\\
&e^{\gamma-1}\,,\quad && a=2\,.
\end{aligned}
\right.
\end{equation*}

\medskip

\begin{theorem}\label{thm1}
(1) The function $g_1$ is strictly increasing on $(-a\,,x_0)$ and
strictly decreasing on $(x_0\,,\infty)$ if and only if $a\leq0$\,;
$g_1$ is strictly decreasing on $(-a\,,x_1)$ and $(x_2\,,\infty)$\,,
and strictly increasing on $(x_1\,,0)$ and $(0\,,x_2)$ if and only if $0<a<1$ or $a>2$\,;
and $g_1$ is strictly decreasing on $(-a,0)$ and $(0,\infty)$ if and only if $1\le a\le 2$\,,
where $x_i$ satisfies $x_i\psi(x_i+a)=\log\Gamma(x_i+a)$\,, $i=0\,,\,1\,,\,2$ and $x_1<0<x_2$\,.

(2) The function $g_1$ is strictly LCM on $(0\,,\infty)$ if and only if $1\leq a\leq2$\,;
and $g_1$ is strictly LCM on $(-a\,,\infty)$ if and only if $a=1$ or $a=2$\,.
\end{theorem}

\begin{remark}
The sufficient condition of the LCM property for $g_1$ on $(0\,,\infty)$ in Theorem \ref{thm1} (2)
can also be obtained by taking $c=0$ in \cite[Remark 2.1]{MQA}.
\end{remark}

\medskip

The following inequalities (\ref{inequality 1}) and (\ref{inequality 2}) can be easily derived from
the monotonicity and logarithmic convexity properties of $g_1$ in Theorem \ref{thm1} (2).

\begin{corollary}
(1) For $0<x<y$\,, the inequality
\begin{align}\label{inequality 1}
\frac{\left(\Gamma(x+a)\right)^{\frac 1x}}{\left(\Gamma(y+a)\right)^{\frac 1y}}
<1
\end{align}
holds for $1\leq a\leq2$\,.

\medskip

(2) For $x,y>0$\,, the inequality
\begin{align}\label{inequality 2}
\frac{\left(\Gamma\left(\frac{x+y}{2}+a\right)\right)^{\frac{2}{x+y}}}
{\sqrt{\left(\Gamma(x+a)\right)^{\frac 1x}\left(\Gamma(y+a)\right)^{\frac 1y}}}
\geq 1
\end{align}
holds for $1\leq a\leq2$\,.
The equality is true if and only if $x=y$\,.
\end{corollary}

\bigskip

\begin{theorem}\label{thm2}
For $a>0\,,c\in\mathbb{R}$\,, let $x\in(0\,,\infty)$ and
$$g_2(x)
\equiv f_{a,c,+1}(x)
=\dfrac{(\Gamma(x+a))^{\frac 1x}}{x^c}\,.$$

(1) The function $g_2$ is strictly decreasing on $(0\,,\infty)$ if and only if $c\geq1$ for $\dfrac12\leq a\leq 1$ or $a\geq2$\,;
$g_2$ is strictly increasing on $(0\,,\infty)$ if and only if $c\leq0$ for $a=1$ or $a=2$\,;
and $g_2$ is strictly increasing on $(0\,,\infty)$ if and only if $c\leq h_2(x_3)$ for $1<a<2$\,,
where $x_3$ satisfies $x_3^2\psi'(x_3+a)+\log\Gamma(x_3+a)=x_3\psi(x_3+a)$
and $h_2(x_3) \equiv x_3\psi'(x_3+a)$\,.

(2) The function $g_2$ is strictly log-convex on $(0\,,\infty)$ if and only if $c\geq1$ for $a\geq2$\,;
and $g_2$ is strictly log-concave on $(0\,,\infty)$ if and only if $c\leq0$ for $a=2$\,.

(3) The function $g_2$ is geometrically convex on $(0\,,\infty)$ if and only if $(a\,,c)\in D_1\cup D_2$\,;
and $g_2$ is geometrically concave on $(0\,,x_3)$ and geometrically convex on $(x_3\,,\infty)$ if and only if $(a,c)\in D_3\cup D_4$\,.
\end{theorem}

\begin{remark}
It is clear that Theorem \ref{thm2} (3) is a generalization of \cite[Theorem 1]{CZZ} which says that
$f_{1,0,+1}$ is geometrically convex on $(0\,,\infty)$\,.
\end{remark}

\medskip

Theorem \ref{thm2} leads to the following corollary.

\medskip

\begin{corollary}
(1) For $0<x<y$\,, the inequality
\begin{align}\label{inequality 4}
\frac{\left(\Gamma(x+a)\right)^{\frac 1x}}{\left(\Gamma(y+a)\right)^{\frac 1y}}
>\left(\frac xy\right)^c
\end{align}
holds for
$(a,c)\in D_5 \cup D_6$\,;
and inequality (\ref{inequality 4}) is reversed for
$(a,c)\in D_7 \cup D_8$\,.

\medskip

(2) For $x,y>0$\,, the inequality
\begin{align}\label{inequality 5}
\frac{\left(\Gamma\left(\frac{x+y}{2}+a\right)\right)^{\frac{2}{x+y}}}
{\sqrt{\left(\Gamma(x+a)\right)^{\frac 1x}\left(\Gamma(y+a)\right)^{\frac 1y}}}
\leq
\left(\frac{x+y}{2\sqrt{xy}}\right)^c
\end{align}
holds for $(a,c)\in D_6$\,;
and inequality (\ref{inequality 5}) is reversed for $(a,c)\in D_8$\,.
The equalities are true if and only if $x=y$\,.

\medskip

(3) For $x\,,y>0$\,, the inequalities
\begin{align}\label{inequality 3}
\left(\frac xy\right)^{\frac{y\psi(y+a)-\log\Gamma(y+a)}{y}}
\leq
\frac{\left(\Gamma(x+a)\right)^{\frac 1x}}{\left(\Gamma(y+a)\right)^{\frac 1y}}
\leq
\left(\frac xy\right)^{\frac{x\psi(x+a)-\log\Gamma(x+a)}{x}}
\end{align}
hold for  $(a,c)\in D_1 \cup D_2$\,.
The equalities are true if and only if $x=y$\,.
\end{corollary}

\bigskip

\begin{theorem}\label{thm3}
For $a>0\,,c\in\mathbb{R}$\,, let $x\in(0\,,\infty)$ and
$$g_3(x)
\equiv g_{a,c,+1}(x)
=\dfrac{(\Gamma(x+a))^{\frac 1x}}{(x+a)^c}\,.$$

(1) The function $g_3$ is strictly decreasing on $(0\,,\infty)$ if and only if $c\geq1$ for $a\geq2$\,;
$g_3$ is strictly increasing on $(0\,,\infty)$ if and only if $c\leq\dfrac{\pi^2}{6}-1$ for $a=2$\,;
and $g_3$ is strictly increasing on $(0\,,\infty)$ if and only if $c\leq h_4(x_4)$
for $\dfrac{3+\sqrt{159}}{12}\le a<2$\,,
where $x_4$ satisfies $x_4^2(x_4+a)\psi'(x_4+a)+(x_4+2a)\log\Gamma(x_4+a)=x_4(x_4+2a)\psi(x_4+a)$
and $h_4(x_4) \equiv\dfrac{x_4(x_4+a)\psi(x_4+a)-(x_4+a)\log\Gamma(x_4+a)}{x_4^2}\,.$

(2) The function $g_3$ is geometrically convex on $(0\,,\infty)$ for $(a,c)\in D_9 \cup D_{10}$\,;
and $g_3$ is geometrically concave on $(0\,,x_3)$ for $(a,c)\in D_3$
and geometrically convex on $(x_3\,,\infty)$ for $(a,c)\in D_4$\,,
where $x_3$ is the same as in Theorem \ref{thm2} (1).
\end{theorem}

\medskip

\begin{remark}
The sufficient condition of the decreasing property for $g_3$ in Theorem \ref{thm3} (1)
can also be obtained by \cite[Theorem 1.2]{MQA}.
\end{remark}

\medskip

The following corollary can be directly derived by Theorem \ref{thm3}.

\begin{corollary}
(1) For $0<x<y$\,, the inequality
\begin{align}\label{inequality 6}
\frac{\left(\Gamma(x+a)\right)^{\frac 1x}}{\left(\Gamma(y+a)\right)^{\frac 1y}}
>\left(\frac {x+a}{y+a}\right)^c
\end{align}
holds for $(a,c)\in D_6$\,;
and inequality (\ref{inequality 6}) is reversed for  $(a,c)\in D_{11}$\,.

\medskip

(2) For $x,y>0$\,, the inequalities
\begin{align}\label{inequality 7}
\left(\frac xy\right)^{\frac{y\psi(y+a)-\log\Gamma(y+a)}{y}-\frac{cy}{y+a}}\left(\frac{x+a}{y+a}\right)^c
\leq
\frac{\left(\Gamma(x+a)\right)^{\frac 1x}}{\left(\Gamma(y+a)\right)^{\frac 1y}}
\leq
\left(\frac xy\right)^{\frac{x\psi(x+a)-\log\Gamma(x+a)}{x}-\frac{cx}{x+a}}\left(\frac{x+a}{y+a}\right)^c
\end{align}
hold for $(a,c)\in D_9 \cup D_{10}$\,.
The equalities are true if and only if $x=y$\,.
\end{corollary}

\bigskip

\section{Lemmas}

In this section,
we show some lemmas which are needed in the proofs of the main results.
The following formulas will be frequently used in the proofs of lemmas\cite{MI,SM}.

\medskip

Leibniz's Theorem for differentiation of the product of two functions:
\begin{equation}\label{eq1}
(u(x)v(x))^{(n)}=\sum_{k=0}^{n} \left(\begin{array}{c}
                                    n \\
                                    k
                                  \end{array}\right)
u^{(k)}(x)v^{(n-k)}(x)\,.
\end{equation}

\medskip

Recurrence formulas of $\Gamma\,,\psi$\,:
\begin{equation*}
\Gamma(x+1)=x\Gamma(x)\,, \quad \psi(x+1)=\psi(x)+\frac1x\,.
\end{equation*}

\medskip

Special values of $\Gamma\,,\psi\,,\psi'$\,:
\begin{equation*}
\Gamma(1)=1\,,\quad \psi(1)=-\gamma\,, \quad \psi'(2)=\frac{\pi^2}{6}-1\,.
\end{equation*}

\medskip

Asymptotic formulas of $\log\Gamma\,,\,\psi\,,\,\psi'\,,\psi''$\,:
for $x\rightarrow\infty\,\, {\rm with\,\, |arg}\, x|<\pi$\,,
\begin{align*}
\log\Gamma(x)
&\sim
\left(x-\frac12\right)\log x-x+\frac12\log(2\pi)+\frac{1}{12x}-\frac{1}{360x^3}+\cdots\,,\\
\psi(x)
&\sim
\log x-\frac{1}{2x}-\frac{1}{12x^2}+\frac{1}{120x^4}-\cdots\,,\\
\psi'(x)
&\sim
\frac1x+\frac{1}{2x^2}+\frac{1}{6x^3}-\frac{1}{30x^5}+\cdots\,,\\
\psi''(x)
&\sim-\frac{1}{x^2}-\frac{1}{x^3}-\frac{1}{2x^4}+\cdots\,.
\end{align*}

\medskip

For $x\in(0\,,\infty)$\,, the following inequalities of the polygamma functions hold \cite[Theorem 3]{GQ}:
\begin{equation}\label{ineq1}
\frac{(n-1)!}{x^n}+\frac{n!}{2x^{n+1}}
\leq
(-1)^{n+1}\psi^{(n)}(x)
\leq
\frac{(n-1)!}{x^n}+\frac{n!}{x^{n+1}}\,,\quad  n=1,2,\cdots\,.
\end{equation}
Moreover, there holds the identity for $\psi''$ \cite{GMF}:
\begin{equation}\label{eq2}
\psi''(x)
=-\frac{1}{x^2}-\frac{1}{x^3}-\frac{1}{2x^4}+\frac{\theta}{6x^6}\,,\quad 0<\theta<1\,.
\end{equation}

\medskip

\begin{lemma}\label{lem1}\cite[Proposition 4.3]{CPN}\cite[Theorem C]{CZZ}
Let $I\subset(0\,,\infty)$ be an interval. If $f:I\rightarrow(0\,,\infty)$ is a differentiable function,
then the following assertions are equivalent:

(1) The function $f$ is geometrically convex (geometrically concave) on $I$\,;

(2) The function $g(x)\equiv\dfrac{xf'(x)}{f(x)}$ is increasing (decreasing) on $I$\,;

(3) The function $f$ satisfies the inequalities
\begin{align*}
\left(\frac xy\right)^{\frac{y f'(y)}{f(y)}}
\leq(\geq) \frac{f(x)}{f(y)}
\leq(\geq) \left(\frac xy\right)^{\frac{x f'(x)}{f(x)}}\,, \quad
\forall\,x\,,y\in I\,.
\end{align*}
\end{lemma}

\medskip

\begin{lemma}\label{lem6}
For $n=1\,,2\,,\cdots$\,, there hold
\begin{align}\label{seq1}
\begin{split}
(\log g_1(0))^{(n)}
&=
\left\{
\begin{aligned}
&-\frac{\psi^{(n)}(1)}{n+1}\,,\quad a=1\,,\\
&-\frac{\psi^{(n)}(2)}{n+1}\,,\quad a=2\,.\\
\end{aligned}
\right.
\end{split}
\end{align}
\end{lemma}

\medskip

\begin{proof}
We first consider the case for $a=1$\,.

When $n=1$\,,
\begin{align*}
(\log g_1(0))'
&=\lim_{x\rightarrow0}\frac{\log g_1(x)-\log g_1(0)}{x-0}\\
&=\lim_{x\rightarrow0}\frac{-\log\Gamma(x+1)-x\gamma}{x^2}
=-\frac{\psi'(1)}{2}\,.
\end{align*}

We assume that (\ref{seq1}) holds when $n=k$ ($k\in\mathbb{Z}\,,k>1$)\,.

Then by L'Hopital Rule, we get
\begin{align*}
(\log g_1(0))^{(k+1)}
&=\lim_{x\rightarrow0}\frac{(\log g_1(x))^{(k)}-(\log g_1(0))^{(k)}}{x-0}\\
&=\lim_{x\rightarrow0}\frac{(-1)^kk!\delta_k(x)+\frac{\psi^{(k)}(1)}{k+1}x^{k+1}}{x^{k+2}}\\
&=\lim_{x\rightarrow0}\frac{-\psi^{(k+1)}(x+1)}{k+2}=-\frac{\psi^{(k+1)}(1)}{k+2}\,,
\end{align*}
where $$\delta_k(x)\equiv -\log\Gamma(x+a)-\sum\limits_{j=1}^{k}\dfrac{(-1)^jx^j}{j!}\psi^{(j-1)}(x+a)\,.$$

By induction, (\ref{seq1}) holds for $n=1\,,2\,,\cdots$ when $a=1$\,.

\medskip

In a similar way, we can prove that (\ref{seq1}) holds for $n=1\,,2\,,\cdots$ when $a=2$\,.

\medskip

The proof is complete.
\end{proof}

\medskip

\begin{lemma}\label{lem2}
For $a\in\mathbb{R}$\,, let $x\in(-a\,,\infty)$ and
\begin{align*}
h_1(x)
\equiv -x\psi(x+a)+\log\Gamma(x+a)\,.
\end{align*}

(1) The function $h_1$ is strictly decreasing from $(-a\,,\infty)$ onto $(-\infty\,,\infty)$ if and only if $a\leq0$\,.

(2) The function $h_1$ is strictly increasing on $(-a\,,0]$
and strictly decreasing on $[0\,,\infty)$
with $h_1(x)\in (-\infty\,,\log\Gamma(a)]$ if and only if $a>0$\,.
Moreover, $h_1(x)<0$ on $(-a\,,x_1)\cup(x_2\,,\infty)$\,,
$h_1(x)>0$ on $(x_1\,,x_2)$ for $0<a<1$ or $a>2$\,;
and $h_1(x)<0$ on $(-a\,,0)\cup(0\,,\infty)$ for $1\leq a\leq2$\,,
where $x_1\,,x_2$ are the same as in Theorem \ref{thm1} (1).
\end{lemma}

\medskip

\begin{proof}
Let $t=x+a$\,. Then
\begin{align*}
{h}_1(x)={h}_1(t-a)\equiv\widetilde{h}_1(t)
=-(t-a)\psi(t)+\log\Gamma(t)\,,\quad t\in(0\,,\infty)\,.
\end{align*}
It suffices to study the monotonicity property and the range of $\widetilde{h}_1$\,.

\medskip

We first prove the monotonicity property of $\widetilde{h}_1$\,.

\medskip

Differentiation yields
\begin{align*}
\widetilde{h}_1'(t)=-(t-a)\psi'(t)\,.
\end{align*}
Therefore $\widetilde{h}_1$ is strictly decreasing on $(0\,,\infty)$ if and only if $a\leq0$\,;
and $\widetilde{h}_1$ is strictly increasing on $(0\,,a]$ and
strictly decreasing on $[a\,,\infty)$ if and only if $a>0$\,.

\medskip

Then we calculate the range of $\widetilde{h}_1$.

\medskip

By the asymptotic formulas of $\log\Gamma\,,\,\psi$\,, we get
\begin{align*}
\lim_{t\rightarrow\infty} \widetilde{h}_1(t)
&=\lim_{t\rightarrow\infty} \left(-(t-a)\left(\log t-\frac{1}{2t}+O\left(\frac{1}{t^2}\right)\right)+\left(t-\frac12\right)\log t\right.\\
&\quad\quad\quad \left.-t+\frac12\log (2\pi)+O\left(\frac 1t\right)\right)\\
&=\lim_{t\rightarrow\infty} \left(t\left(\left(a-\frac12\right)\frac{\log t}{t}-1\right)
+\frac{t-a}{2t}+\frac12\log (2\pi)+O\left(\frac 1t\right)\right)\\
&=-\infty\,,\quad a\in\mathbb{R}\,.
\end{align*}

\medskip

By the recurrence formulas of $\Gamma\,,\,\psi$\,, we get
\begin{align*}
\lim\limits_{t\rightarrow0^+} \widetilde{h}_1(t)
&=\lim\limits_{t\rightarrow0^+} \left(-(t-a)\psi(t+1)+\log\Gamma(t)+\frac{t-a}{t}\right)
=\infty\,,\quad a\leq0
\end{align*}
and
\begin{align*}
\lim_{t\rightarrow0^+} \widetilde{h}_1(t)
&=\lim_{t\rightarrow0^+} \frac1t \left(t(-(t-a)\psi(t+1)+\log\Gamma(t+1))-t\log t+t-a\right)
=-\infty\,,\quad a>0\,.
\end{align*}

The limiting value $\lim\limits_{t\rightarrow a}\widetilde{h}_1(t)=\log\Gamma(a)$ is clear for $a>0$\,.

Therefore $\widetilde{h}_1(t)\in (-\infty\,,\infty)$ for $a\leq0$\,;
and $\widetilde{h}_1(t)\in (-\infty\,,\log\Gamma(a)]$ for $a>0$\,.
Moreover,
for $0<a<1$ or $a>2$\,, there exist $t_1\,,\,t_2$ such that
$\widetilde{h}_1(t)<0$ on $(0\,,t_1)\cup(t_2\,,\infty)$ and $\widetilde{h}_1(t)>0$ on $(t_1\,,t_2)$\,;
and for $1\leq a\leq2$\,, $\widetilde{h}_1(t)<0$ on $(0\,,a)\cup(a\,,\infty)$\,,
where $t_i$ satisfies $(t_i-a)\psi(t_i)=\log\Gamma(t_i)\,, i=1,2$ and $t_1<a<t_2$\,.

\medskip

The proof is complete.
\end{proof}

\medskip

\begin{lemma}\label{lem3}
For $a>0$\,, let $x\in(0\,,\infty)$ and
\begin{equation*}
h_2(x)
\equiv \frac{x\psi(x+a)-\log\Gamma(x+a)}{x}\,.
\end{equation*}

(1) The function
$h_2$ is strictly increasing on $(0\,,\infty)$ if and only if $\dfrac12\leq a\leq 1$ or $a\geq2$\,.
Moreover, $h_2(x)\in(-\infty\,,1)$ for $\dfrac12\leq a<1$ or $a>2$\,;
and $h_2(x)\in(0\,,1)$ for $a=1$ or $a=2$\,.

(2) The function $h_2$ is strictly decreasing on $(0,x_3]$ and strictly increasing on $[x_3,\infty)$
with $h_2(x)\in[h_2(x_3),\infty)$ if and only if $1<a<2$\,,
where $x_3$ is the same as in Theorem \ref{thm2} (1).
\end{lemma}

\medskip

\begin{proof}
Let $t=x+a$\,. Then
\begin{equation*}
h_2(x)=h_2(t-a)\equiv \widetilde{h}_2(t)
=-\frac{\widetilde{h}_1(t)}{t-a}\,,\quad  t\in(a\,,\infty)\,,
\end{equation*}
where $\widetilde{h}_1(t)$ is the same as in the proof of Lemma \ref{lem2}.
It suffices to study the monotonicity property and the range of $\widetilde{h}_2$\,.

\medskip

We first prove the monotonicity property of $\widetilde{h}_2$\,.

\medskip

Differentiation gives
\begin{align*}
\widetilde{h}_2'(t)
&\equiv \frac{h_{21}(t)}{(t-a)^2}\,,
\end{align*}
where
$$h_{21}(t)\equiv (t-a)^2\psi'(t)-(t-a)\psi(t)+\log\Gamma(t)\,.$$

It is easy to obtain
\begin{align*}
h_{21}'(t)
&=(t-a)((t-a)\psi''(t)+\psi'(t))\,.
\end{align*}

By the inequality (\ref{ineq1}) of $\psi'$ and the identity (\ref{eq2}) of $\psi''$\,, we get
\begin{align*}
(t-a)\psi''(t)+\psi'(t)
&> (t-a)\left(-\frac{1}{t^2}-\frac{1}{t^3}-\frac{1}{2t^4}\right)+\frac 1t+\frac{1}{2t^2}\\
&=\frac{1}{2t^4}\left((2a-1)t^2+(2a-1)t+a\right)\,.
\end{align*}

Since $(2a-1)t^2+(2a-1)t+a>0$ on $(a\,,\infty)$ if and only if $a\geq\dfrac12$\,,
we have that $h_{21}$ is strictly increasing on $(a\,,\infty)$ for $a\geq\dfrac12$
and hence
\begin{equation*}
h_{21}(t)
>\lim_{t\rightarrow a^+}h_{21}(t)=\log\Gamma(a)\,.
\end{equation*}
Thus $h_{21}(t)>0$ on $(a\,,\infty)$ for $\dfrac12\leq a\leq1$ or $a\geq2$\,.

\medskip

For $0<a<\dfrac12$\,, the limiting value $\lim\limits_{t\rightarrow a^+}h_{21}(t)=\log\Gamma(a)>0$ is clear.
By the asymptotic formulas of $\log\Gamma\,,\,\psi\,,\,\psi'$\,, we get

\begin{align*}
\lim_{t\rightarrow\infty} h_{21}(t)
&=\lim_{t\rightarrow\infty}
\left((t-a)^2\left(\frac 1t+\frac{1}{2t^2}+O\left(\frac{1}{t^3}\right)\right)
-(t-a)\left(\log t-\frac{1}{2t}+O\left(\frac{1}{t^2}\right)\right)\right.\\
&\left.\quad\quad\quad +\left(t-\frac12\right)\log t-t+\frac12\log (2\pi)+O\left(\frac 1t\right)\right)\\
&=\lim_{t\rightarrow\infty}
\left(\frac{-2at+a^2}{t}+\left(a-\frac12\right)\log t+\frac{(t-a)^2}{2t^2}+\frac{t-a}{2t}+\frac12\log (2\pi)+O\left(\frac 1t\right)\right)\\
&=-\infty\,.
\end{align*}

For $1<a<2$\,, we have
$\lim\limits_{t\rightarrow a^+}h_{21}(t)=\log\Gamma(a)<0$
and $\lim\limits_{t\rightarrow\infty}h_{21}(t)=\infty$\,.

\medskip

Therefore $h_{21}(t)>0$ on $(a\,,\infty)$ and hence
$\widetilde{h}_2$ is strictly increasing on $(a\,,\infty)$
if and only if $\dfrac12\leq a\leq 1$ or $a\geq2$\,.

\medskip

Since $h_{21}$ is strictly increasing on $(a\,,\infty)$ for $1<a<2$\,, there exists $t_3\in(a\,,\infty)$
such that $h_{21}(t)<0$ on $(a\,,t_3)$ and $h_{21}(t)>0$ on $(t_3\,,\infty)$\,,
where $t_3$ satisfies $(t_3-a)^2\psi'(t_3)+\log\Gamma(t_3)=(t_3-a)\psi(t_3)$\,.

\medskip

Therefore $\widetilde{h}_2$ is strictly decreasing on $(a\,,t_3]$
and strictly increasing on $[t_3\,,\infty)$ if and only if $1<a<2$\,.

\medskip

Then we calculate the range of $\widetilde{h}_2$.

\medskip

By the proof of Lemma \ref{lem2}, we have $\lim\limits_{t\rightarrow\infty}\widetilde{h}_1(t)=-\infty$\,.
Using L'Hopital Rule and the asymptotic formula of $\psi'$, we obtain
\begin{align*}
\lim_{t\rightarrow\infty} \widetilde{h}_2(t)
=\lim_{t\rightarrow\infty} (t-a)\left(\frac 1t+O\left(\frac{1}{t^2}\right)\right)
=1\,,\quad a>0\,.
\end{align*}

Calculation yields the limiting value
$$
\lim\limits_{t\rightarrow a^+}\widetilde{h}_2(t)=
\left\{
\begin{aligned}
&-\infty\,, \quad &&\frac12\leq a<1 \,\,\hbox{or}\,\, a>2\,,\\
&0\,, \quad &&a=1 \,\,\hbox{or}\,\, a=2\,,\\
&\infty\,, \quad && 1<a<2\,.
\end{aligned}
\right.
$$

\medskip

Therefore $\widetilde{h}_2(t)\in(-\infty\,,1)$ for $\dfrac12\leq a<1$ or $a>2$\,;
$\widetilde{h}_2(t)\in(0\,,1)$ for $a=1$ or $a=2$\,;
and $\widetilde{h}_2(t)\in[\widetilde{h}_2(t_3)\,,\infty)$ for $1<a<2$\,.

\medskip

The proof is complete.
\end{proof}

\medskip

\begin{openproblem}
What is the monotonicity property of $h_2$ on $(0\,,\infty)$ for $0<a<\dfrac12$\,?
\end{openproblem}

\medskip

\begin{lemma}\label{lem4}
For $a>0$\,, let $x\in(0\,,\infty)$ and
\begin{equation*}
h_3(x) \equiv \dfrac{-x^2\psi'(x+a)+2x\psi(x+a)-2\log\Gamma(x+a)}{x}\,.
\end{equation*}
Then the function $h_3$ is strictly increasing on $(0\,,\infty)$ for $a\geq2$\,.
Moreover, $h_3(x)\in(-\infty\,,1)$ for $a>2$\,; and $h_3(x)\in(0\,,1)$ for $a=2$\,.
\end{lemma}

\medskip

\begin{proof}
Let $t=x+a$\,. Then
\begin{align*}
h_3(x)=h_3(t-a)\equiv\widetilde{h}_3(t)
=\dfrac{-(t-a)^2\psi'(t)+2(t-a)\psi(t)-2\log\Gamma(t)}{t-a}\,,\quad t\in(a\,,\infty)\,.
\end{align*}
It suffices to study the monotonicity property and the range of $\widetilde{h}_3$\,.

\medskip

We first prove the monotonicity property of $\widetilde{h}_3$\,.

\medskip

Differentiation gives
\begin{align*}
\widetilde{h}_3'(t)
&\equiv \frac{h_{31}(t)}{(t-a)^2}\,,
\end{align*}
where
$$h_{31}(t) \equiv -(t-a)^3\psi''(t)+(t-a)^2\psi'(t)-2(t-a)\psi(t)+2\log\Gamma(t)\,.$$

It is easy to obtain
\begin{align*}
h_{31}'(t)
&=-(t-a)^2\left((t-a)\psi'''(t)+2\psi''(t)\right)\,.
\end{align*}

\medskip

By the inequalities (\ref{ineq1}) of $\psi''$ and $\psi'''$\,, we have
\begin{align*}
(t-a)\psi'''(t)+2\psi''(t)
&\leq (t-a)\left(\frac{2}{t^3}+\frac{6}{t^4}\right)+2\left(-\frac{1}{t^2}-\frac{1}{t^3}\right)\\
&= \frac{1}{t^4}\left(2(2-a)t-6a\right)\,.
\end{align*}

\medskip

Since $2(2-a)t-6a<0$ on $(a\,,\infty)$ if and only if $a\geq2$\,,
we have that $h_{31}$ is strictly increasing on $(a\,,\infty)$ for $a\geq2$ and hence
\begin{align*}
h_{31}(t)
>\lim_{t\rightarrow a^+}h_{31}(t)=2\log\Gamma(a)\,.
\end{align*}
Therefore $h_{31}(t)>0$ on $(a\,,\infty)$ and hence
$\widetilde{h}_3$ is strictly increasing on $(a\,,\infty)$ for $a\geq2$\,.

\medskip

Then we calculate the range of $\widetilde{h}_3$.

\medskip

By the asymptotic formulas of $\log\Gamma\,,\,\psi\,,\,\psi'$\,, we get
\begin{align*}
&\lim_{t\rightarrow\infty}
\left(-(t-a)^2\psi'(t)+2(t-a)\psi(t)-2\log\Gamma(t)\right)\\
=&\lim_{t\rightarrow\infty}
\left(-(t-a)^2\left(\frac 1t+\frac{1}{2t^2}+O\left(\frac{1}{t^3}\right)\right)
+2(t-a)\left(\log t-\frac{1}{2t}+O\left(\frac{1}{t^2}\right)\right)\right.\\
&\left.\quad\quad -2\left(\left(t-\frac12\right)\log t-t+\frac12\log(2\pi)+O\left(\frac{1}{t}\right)\right)\right)\\
=&\lim_{t\rightarrow\infty}
\left(t\left((1-2a)\frac{\log t}{t}-\frac{(t-a)^2}{t^2}+2\right)-\frac{(t-a)^2}{2t^2}-\frac{t-a}{t}-\log(2\pi)+O\left(\frac{1}{t}\right)\right)\\
=&\infty\,.
\end{align*}
L'Hopital Rule and the asymptotic formula of $\psi''$ yield
\begin{align*}
\lim_{t\rightarrow\infty} \widetilde{h}_3(t)
=\lim_{t\rightarrow\infty} \left(-(t-a)^2\left(-\frac{1}{t^2}+O\left(\frac{1}{t^3}\right)\right)\right)=1\,.
\end{align*}

It is easy to obtain
$$
\lim\limits_{t\rightarrow a^+}\widetilde{h}_3(t)=
\left\{
\begin{aligned}
&-\infty\,, \quad &&a>2\,,\\
&0\,, \quad &&a=2\,.
\end{aligned}
\right.
$$

Therefore $\widetilde{h}_3(t)\in(-\infty\,,1)$ for $a>2$\,; and $\widetilde{h}_3(t)\in(0\,,1)$ for $a=2$\,.
\end{proof}

\medskip

\begin{openproblem}
What is the monotonicity property of $h_3$ on $(0\,,\infty)$ for $0<a<2$\,?
\end{openproblem}

\medskip

\begin{lemma}\label{lem5}
For $a>0$\,, let $x\in(0\,,\infty)$ and
\begin{equation*}
h_4(x) \equiv\dfrac{x(x+a)\psi(x+a)-(x+a)\log\Gamma(x+a)}{x^2}\,.
\end{equation*}

(1) The function $h_4$ is strictly increasing on $(0\,,\infty)$ for $a\geq2$\,.
Moreover, $h_4(x)\in(-\infty\,,1)$ for $a>2$\,; and $h_4(x)\in\left(\dfrac{\pi^2}{6}-1\,,1\right)$ for $a=2$\,.

(2) The function
$h_4$ is strictly decreasing on $(0,x_4]$ and strictly increasing on $[x_4,\infty)$ with $h_4(x)\in[h_4(x_4)\,,\infty)$
for $\dfrac{3+\sqrt{159}}{12}\le a<2$\,,
where $x_4$ is the same as in Theorem \ref{thm3} (1).
\end{lemma}

\medskip

\begin{proof}
Let $t=x+a$\,. Then
\begin{align*}
{h}_4(x)={h}_4(t-a)\equiv\widetilde{h}_4(t)
=\frac{t(t-a)\psi(t)-t\log\Gamma(t)}{(t-a)^2}\,,\quad t\in (a\,,\infty)\,.
\end{align*}
It suffices to study the monotonicity property and the range of $\widetilde{h}_4$\,.

\medskip

We first prove the monotonicity property of $\widetilde{h}_4$\,.

\medskip

Differentiation gives
\begin{align*}
\widetilde{h}_4'(t)
&\equiv \frac{h_{41}(t)}{(t-a)^3}\,,
\end{align*}
where
$$h_{41}(t)\equiv t(t-a)^2\psi'(t)-(t^2-a^2)\psi(t)+(t+a)\log\Gamma(t)\,.$$

\medskip

It is easy to obtain
\begin{align*}
h_{41}'(t)
&=t(t-a)^2\psi''(t)+2(t-a)^2\psi'(t)-(t-a)\psi(t)+\log\Gamma(t)
\end{align*}
and
\begin{align*}
h_{41}''(t)
&=(t-a)\left(t(t-a)\psi'''(t)+(5t-3a)\psi''(t)+3\psi'(t)\right)\,.
\end{align*}

\medskip

By the inequalities (\ref{ineq1}) of $\psi'\,,\psi'''$ and the identity (\ref{eq2}) of $\psi''$\,, we get
\begin{align*}
&\quad t(t-a)\psi'''(t)+(5t-3a)\psi''(t)+3\psi'(t)\\
&>t(t-a)\left(\frac{2}{t^3}+\frac{3}{t^4}\right)+(5t-3a)\left(-\frac{1}{t^2}-\frac{1}{t^3}-\frac{1}{2t^4}\right)
+3\left(\frac{1}{t}+\frac{1}{2t^2}\right)\\
&=\frac{1}{2t^4}\left((2a-1)t^2-5t+3a\right)\,.
\end{align*}

Since $(2a-1)t^2-5t+3a\geq0$ on $(a\,,\infty)$ if and only if $a\geq\dfrac{3+\sqrt{159}}{12}\approx1.3$\,,
we have that $h_{41}'$ is strictly increasing on $(a\,,\infty)$ for $a\geq\dfrac{3+\sqrt{159}}{12}$ and hence
\begin{align*}
h_{41}'(t)>\lim\limits_{t\rightarrow a^+}h_{41}'(t)=\log\Gamma(a)\,.
\end{align*}
Moreover, $h_{41}'(t)>0$
and hence $h_{41}$ is strictly increasing on $(a\,,\infty)$ for $a\geq2$\,.

Then for $a\geq2$\,, we have
$$h_{41}(t)>\lim\limits_{t\rightarrow a^+}h_{41}(t)=2a\log\Gamma(a)\,.$$
Therefore $h_{41}(t)>0$ and hence
$\widetilde{h}_4$ is strictly increasing on $(a\,,\infty)$ for $a\geq2$\,.

\medskip

We consider the case for $\dfrac{3+\sqrt{159}}{12}\le a<2$ in the following.

\medskip

The limiting value
$\lim\limits_{t\rightarrow a^+}h_{41}'(t)=\log\Gamma(a)<0$ is clear.

By the asymptotic formulas of $\log\Gamma\,,\,\psi\,,\,\psi'\,,\psi''$\,, we get

\begin{align*}
\lim_{t\rightarrow\infty} h_{41}'(t)
&=\lim_{t\rightarrow\infty} \left( t(t-a)^2\left(-\frac{1}{t^2}-\frac{1}{t^3}+O\left(\frac{1}{t^4}\right)\right)
+2(t-a)^2\left(\frac{1}{t}+\frac{1}{2t^2}+O\left(\frac{1}{t^3}\right)\right) \right.\\
&\left.\quad\quad\quad  -(t-a)\left(\log t-\frac{1}{2t}+O\left(\frac{1}{t^2}\right)\right)
+\left(t-\frac12\right)\log t-t+\frac12\log (2\pi)+O\left(\frac1t\right)\right)\\
&=\lim_{t\rightarrow\infty} \left(\frac{-2at+a^2}{t}+\left(a-\frac12\right)\log t+\frac{t-a}{2t}+\frac12\log (2\pi)+O\left(\frac1t\right)\right)\\
&=\infty\,.
\end{align*}

Since $h_{41}'$ is strictly increasing on $(a\,,\infty)$ for $a\geq\dfrac{3+\sqrt{159}}{12}$\,,
there exists $ \tilde{t}_4\in(a\,,\infty)$ such that
$h'_{41}(t)<0$ on $(a\,,\tilde{t}_4)$ and $h'_{41}(t)>0$ on $(\tilde{t}_4\,,\infty)$\,,
where $\tilde{t}_4$ satisfies $\tilde{t}_4(\tilde{t}_4-a)^2\psi''(\tilde{t}_4)+2(\tilde{t}_4-a)^2\psi'(\tilde{t}_4)+\log\Gamma(\tilde{t}_4)=(\tilde{t}_4-a)\psi(\tilde{t}_4)$\,.

Hence $h_{41}$ is strictly decreasing on $(a\,,\tilde{t}_4]$ and strictly increasing on $[\tilde{t}_4\,,\infty)$
for $\dfrac{3+\sqrt{159}}{12}\le a<2$\,.

\medskip

The limit values $\lim\limits_{t\rightarrow a^+}h_{41}(t)=2a\log\Gamma(a)<0$ is clear.

By the asymptotic formulas of $\log\Gamma\,,\,\psi\,,\,\psi'$\,, we get
\begin{align*}
\lim_{t\rightarrow\infty} h_{41}(t)
&=\lim_{t\rightarrow\infty} \left(t(t-a)^2\left(\frac1t+\frac{1}{2t^2}+\frac{1}{6t^3}+O\left(\frac{1}{t^5}\right)\right)
-(t^2-a^2)\left(\log t-\frac{1}{2t}-\frac{1}{12t^2}\right.\right.\\
&\quad\quad\quad \left.\left.+O\left(\frac{1}{t^4}\right)\right)
+(t+a)\left(\left(t-\frac12\right)\log t-t+\frac12\log 2\pi+\frac{1}{12t}+O\left(\frac{1}{t^3}\right)\right)\right)\\
&=\lim_{t\rightarrow\infty} \left(t\left(\left(a-\frac12\right)\log t
+\left(a^2-\frac{a}{2}\right)\frac{\log t}{t}+1+\frac12\log (2\pi)-3a\right)\right.\\
&\quad\quad\quad \left.+a\left(a+\frac12\log (2\pi)-1\right)+\frac{4t^2-3at+a^2}{12t^2}
+O\left(\frac{1}{t^2}\right)\right)\\
&=\infty\,.
\end{align*}

By the monotonicity property of $h_{41}$\,,
there exists $t_4\,(>\tilde{t}_4)$ such that $h_{41}(t)<0$ on $(a\,,t_4)$ and $h_{41}(t)>0$ on $(t_4\,,\infty)$\,,
where $t_4$ satisfies $t_4(t_4-a)^2\psi'(t_4)+(t_4+a)\log\Gamma(t_4)=(t_4^2-a^2)\psi(t_4)$\,.

Therefore $\widetilde{h}_4$ is strictly decreasing on $(a\,,t_4]$ and
strictly increasing on $[t_4\,,\infty)$ for $\dfrac{3+\sqrt{159}}{12}\le a<2$.

\medskip

Then we calculate the range of $\widetilde{h}_4$\,.

\medskip

The following limiting values
$$
\lim\limits_{t\rightarrow\infty}\widetilde{h}_4(t)
=\lim\limits_{t\rightarrow\infty}\dfrac{t}{t-a}\widetilde{h}_2(t)
=1\,,\quad a>0\,,
$$
and
$$
\lim\limits_{t\rightarrow a^+}\widetilde{h}_4(t)=
\left\{
\begin{aligned}
&\infty\,,
\quad &&\dfrac{3+\sqrt{159}}{12}\le a<2\,,\\
&-\infty\,, \quad &&a>2\
\end{aligned}
\right.
$$
are clear.

For $a=2$\,, by L'Hopital Rule, we get
\begin{align*}
\lim_{t\rightarrow2^+}\widetilde{h}_4(t)
&=\psi'(2)=\frac{\pi^2}{6}-1\,.
\end{align*}

Therefore $\widetilde{h}_4(t)\in(-\infty\,,1)$ for $a>2$\,;  $\widetilde{h}_4(t)\in\left(\dfrac{\pi^2}{6}-1\,,1\right)$ for $a=2$\,;
and $\widetilde{h}_4(t)\in[\widetilde{h}_4(t_4)\,,\infty)$ for $\dfrac{3+\sqrt{159}}{12}\le a<2$\,.

\medskip

The proof is complete.
\end{proof}

\medskip

\begin{openproblem}
What is the monotonicity property of $h_4$ on $(0\,,\infty)$ for $0<a<\dfrac{3+\sqrt{159}}{12}$\,?
\end{openproblem}

\bigskip

\section{Proofs of main results}

\begin{proof}[Proof of Theorem \ref{thm1}]
(1) Logarithmic differentiation gives
\begin{align*}
\frac{g'_1(x)}{g_1(x)}
&\equiv \frac{h_1(x)}{x^2}\,,
\end{align*}
where $h_1(x)$ is the same as in Lemma \ref{lem2}.

\medskip

By Lemma \ref{lem2}, we have that there exists $x_0\in(-a\,,\infty)$
such that $g_1$ is strictly increasing on $(-a\,,x_0)$
and strictly decreasing on $(x_0\,,\infty)$ if and only if $a\leq0$\,,
where $x_0$ satisfies $x_0\psi(x_0+a)=\log\Gamma(x_0+a)$\,;
$g_1$ is strictly decreasing on $(-a\,,x_1)$\,,\,$(x_2\,,\infty)$\,,
and strictly increasing on $(x_1\,,0)$\,,\,$(0\,,x_2)$ if and only if $0<a<1$ or $a>2$\,,
where $x_i$ satisfies $x_i\psi(x_i+a)=\log\Gamma(x_i+a)$\,, $i=1,2$\,;
and $g_1$ is strictly decreasing on
$(-a,0)$ and $(0,\infty)$ if and only if $1\le a\le 2$\,.

\medskip

(2) By (1), we have that $g_1$ is not LCM on $(-a\,,\infty)$ or $(0\,,\infty)$ for $a\leq0$\,,
$0<a<1$ or $a>2$\,.
Therefore we only need to consider the LCM property of $g_1$ for $1\leq a\leq2$\,.

\medskip

For $x\in(-a\,,\infty) \backslash \{0\}$\,, by the formula (\ref{eq1}), we get
\begin{align*}
(-1)^n(\log g_1(x))^{(n)}
&=(-1)^{n+1}\left(\frac{(-1)^n n!}{x^{n+1}}\log\Gamma(x+a)+\sum_{k=1}^{n}\frac{(-1)^{n-k}n!}{k!x^{n-k+1}}\psi^{(k-1)}(x+a)\right)\\
&\equiv \frac{n!}{x^{n+1}}\delta_n(x)\,,
\end{align*}
where $\delta_n(x)\equiv-\log\Gamma(x+a)-\sum\limits_{k=1}^{n}\dfrac{(-1)^kx^k}{k!}\psi^{(k-1)}(x+a)$
is the same as in the proof of Lemma \ref{lem6}.

By differentiation, we get
\begin{align*}
\delta_n'(x)
&=\frac{(-1)^{n+1}x^n}{n!}\psi^{(n)}(x+a)
=\sum_{k=0}^{\infty}\frac{x^n}{(k+x+a)^{n+1}}
\end{align*}
and hence
\begin{align*}
\delta_n'(x)
\left\{
\begin{aligned}
&<0\,,\quad x\in(-a\,,0)\,,\quad  &&\hbox{if}\,\, n \,\,\hbox{is odd}\,,\\
&>0\,,\quad x\in(0\,,\infty)\,,\quad  &&\hbox{if}\,\, n\,\, \hbox{is odd}\,,\\
&>0\,,\quad x\in(-a\,,\infty) \backslash \{0\}\,,\quad  &&\hbox{if}\,\, n\,\, \hbox{is even}\,.
\end{aligned}
\right.
\end{align*}

For $n$ is odd and $x\in(-a\,,\infty) \backslash \{0\}$\,, we have
\begin{align*}
\delta_n(x)
>\lim\limits_{x\rightarrow0}\delta_n(x)
=-\log\Gamma(a)\,.
\end{align*}
Then $\delta_n(x)>0$ and hence
$(-1)^n(\log g_1(x))^{(n)}>0$ on $(-a\,,\infty) \backslash \{0\}$ if and only if $1\leq a\leq2$\,.

\medskip

For $n$ is even and $x\in(-a\,,0)$\,, we have
\begin{align*}
\delta_n(x)
<\lim\limits_{x\rightarrow0}\delta_n(x)
=-\log\Gamma(a)\,.
\end{align*}
Then $\delta_n(x)<0$ and hence
$(-1)^n(\log g_1(x))^{(n)}>0$ on $(-a\,,0)$ if and only if $0<a\leq1$ or $a\geq2$\,.

\medskip

For $n$ is even and $x\in(0\,,\infty)$\,, we have
\begin{align*}
\delta_n(x)
>\lim\limits_{x\rightarrow0}\delta_n(x)
=-\log\Gamma(a)\,.
\end{align*}
Then $\delta_n(x)>0$ and hence
$(-1)^n(\log g_1(x))^{(n)}>0$ on $(0\,,\infty)$ if and only if $1\leq a\leq2$\,.

\medskip

Therefore $g_1$ is strictly LCM on $(0\,,\infty)$ if and only if $1\leq a\leq2$\,.

\medskip

(2)
By Lemma \ref{lem6}, we get
\begin{align*}
(-1)^n(\log g_1(0))^{(n)}
&=
\left\{
\begin{aligned}
&\frac{(-1)^{n+1}\psi^{(n)}(1)}{n+1}\,,\quad a=1\,,\\
&\frac{(-1)^{n+1}\psi^{(n)}(2)}{n+1}\,,\quad a=2\,,\\
\end{aligned}
\right.\\
\end{align*}
which are clearly positive.

Together with the proof in (1), we have that $g_1$ is strictly LCM on $(-a\,,\infty)$
if and only if $a=1$ or $a=2$\,.

\medskip

The proof is complete.
\end{proof}

\medskip

\begin{proof}[Proof of Theorem \ref{thm2}]
(1) Logarithmic differentiation leads to
\begin{equation}\label{deq3}
\begin{split}
\frac {g_2'(x)}{g_2(x)}
& \equiv \frac{h_2(x)-c}{x}\,,
\end{split}
\end{equation}
where $h_2(x)$ is the same as in
Lemma \ref{lem3}.

\medskip

By Lemma \ref{lem3}, we have that $g_2$ is strictly decreasing on $(0\,,\infty)$
if and only if $c\geq1$ for $\dfrac12\leq a\leq 1$ or $a\geq2$\,;
$g_2$ is strictly increasing on $(0\,,\infty)$ if and only if $c\leq0$ for $a=1$ or $a=2$\,;
and $g_2$ is strictly increasing on $(0\,,\infty)$ if and only if $c\leq h_2(x_3)$ for $1<a<2$\,,
where $x_3$ satisfies $x_3^2\psi'(x_3+a)+\log\Gamma(x_3+a)=x_3\psi(x_3+a)$\,.

\medskip

(2) Differentiation gives
\begin{align*}
(\log g_2(x))''
& \equiv \frac{c-h_3(x)}{x^2}\,,
\end{align*}
where $h_3(x)$ is the same as in Lemma \ref{lem4}.

\medskip

By Lemma \ref{lem4},
we have that $g_2$ is strictly log-convex on $(0\,,\infty)$ if and only if $c\geq1$ for $a\geq2$\,;
and $g_2$ is strictly log-concave on $(0\,,\infty)$ if and only if $c\leq0$ for $a=2$\,.

\medskip

(3) By \eqref{deq3}, it is easy to obtain
\begin{align*}
x\frac{g_2'(x)}{g_2(x)}
\equiv h_2(x)-c\,.
\end{align*}

By Lemma \ref{lem1} and Lemma \ref{lem3},
we have that $g_2$ is geometrically convex on $(0\,,\infty)$ if and only if $(a\,,c)\in D_1\cup D_2$\,;
and $g_2$ is geometrically concave on $(0\,,x_3)$ and geometrically convex on
$(x_3\,,\infty)$ if and only if $(a,c)\in D_3\cup D_4$\,.

\medskip

The proof is complete.
\end{proof}

\medskip

\begin{proof}[Proof of Theorem \ref{thm3}]
(1) Logarithmic differentiation gives
\begin{equation}\label{deq2}
\begin{split}
\frac {g_3'(x)}{g_3(x)}
& \equiv \frac{h_4(x)-c}{x+a}\,,
\end{split}
\end{equation}
where $h_4(x)$ is the same as in Lemma \ref{lem5}.

\medskip

By Lemma \ref{lem5}, we have that $g_3$ is strictly decreasing on $(0\,,\infty)$ if and only if $c\geq1$ for $a\geq2$\,;
$g_3$ is strictly increasing on $(0\,,\infty)$ if and only if $c\leq\dfrac{\pi^2}{6}-1$ for $a=2$\,;
and $g_3$ is strictly increasing on $(0\,,\infty)$ if and only if $c\leq h_4(x_4)$ for $\dfrac{3+\sqrt{159}}{12}\le a<2$\,,
where $x_4$ satisfies $x_4^2(x_4+a)\psi'(x_4+a)+(x_4+2a)\log\Gamma(x_4+a)=x_4(x_4+2a)\psi(x_4+a)$\,.

\medskip

(2) By (\ref{deq2}), it is easy to obtain
\begin{align}\label{g3}
\begin{split}
x\frac{g'_3(x)}{g_3(x)}
& \equiv h_2(x)+\frac{ac}{x+a}-c\,,
\end{split}
\end{align}
where $h_2(x)$ is the same as in Lemma \ref{lem3}.

By Lemma \ref{lem1} and Lemma \ref{lem3}, we have that
$g_3$ is geometrically convex on $(0\,,\infty)$ for $(a,c)\in D_9 \cup D_{10}$\,;
and $g_3$ is geometrically concave on $(0\,,x_3)$ for $(a,c)\in D_3$
and geometrically convex on $(x_3\,,\infty)$ for $(a,c)\in D_4$\,.

\medskip

The proof is complete.
\end{proof}

\bigskip

\section{Comparison of Inequalities}

In this section, we compare the inequalities appeared in the corollaries in Section 1.

\begin{remark}
For $c\leq0$ and $x\,,y>0$\,, there holds
\begin{align*}
\left(\frac{x+y}{2\sqrt{xy}}\right)^c
\leq1\,.
\end{align*}

Thus the inequality (\ref{inequality 2}) is better than
the reversed one of inequality (\ref{inequality 5})
for $(a,c)\in D_8$\,.
\end{remark}

\medskip

\begin{remark}
By Lemma \ref{lem3} (1), for $0<x<y$,
we have
\begin{align*}
\left(\frac xy\right)^{\frac{x\psi(x+a)-\log\Gamma(x+a)}{x}}
<1
<\left(\frac xy\right)^c
\end{align*}
for $(a,c)\in D_7 \cup D_8$\,;
and
\begin{align*}
\left(\frac xy\right)^c
<\left(\frac xy\right)^{\frac{y\psi(y+a)-\log\Gamma(y+a)}{y}}
\end{align*}
for $(a,c)\in D_5 \cup D_6$\,.

\medskip

Thus the right side of the inequalities (\ref{inequality 3}) is better than
the inequality (\ref{inequality 1}) and the reversed one of inequality (\ref{inequality 4})
for $(a,c)\in D_7 \cup D_8$\,;
and the left side of the inequalities (\ref{inequality 3}) is better than the inequality (\ref{inequality 4})
for $(a,c)\in D_5 \cup D_6$\,.
\end{remark}

\medskip

\begin{remark}
By \eqref{g3} and Lemma \ref{lem3} (1), it is easy to obtain
\begin{align*}
\lim_{x\rightarrow0^+} x\frac{g'_3(x)}{g_3(x)}
=0\qquad
\hbox{and}
\qquad
\lim_{x\rightarrow\infty} x\frac{g'_3(x)}{g_3(x)}
=1-c\,,
\end{align*}
and hence for $0<x<y$, there holds
\begin{align*}
\left(\frac xy\right)^{\frac{x\psi(x+a)-\log\Gamma(x+a)}{x}-\frac{cx}{x+a}}
<1
\end{align*}
for $(a,c)\in D_8$\,.

\medskip

Thus the right side of the inequalities (\ref{inequality 7}) is better than
the reversed one of inequality (\ref{inequality 6}) for $(a,c)\in D_8$\,.
\end{remark}

\medskip

{\bf Acknowledgements.}
This research was supported by National Natural Science Foundation of China
(NNSFC) under Grant Nos. 11771400 and 11601485.

\medskip




\begin{thebibliography}{HIMSZ}


\bibitem{PML}
Montel P.
Sur les fonctions convexes et les fonctions sousharmoniques[J].
Journal de Math\'{e}matiques Pures et Appliqu\'{e}es,
1928, 9(7): 29-60.

\bibitem{CPN}
Niculescu C P.
Convexity according to the geometric mean[J].
Mathematical Inequalities and Applications,
2000, 3(2): 155-167.

\bibitem{ATR}
Atanassov R D, Tsoukrovski U V.
Some properties of a class of logarithmically completely monotonic functions[J].
Comptes Rendus de I'Acad\'{e}mie Bulgare des Sciences,
1988, 41(2): 21-23.

\bibitem{FP}
Qi F, Chen C P.
A complete monotonicity property of the gamma function[J].
Journal of Mathematical Analysis and Applications,
2004, 296(2): 603-607.

\bibitem{GAI}
Grinshpan A Z, Ismail M E H.
Completely monotonic functions involving the gamma and $q$-gamma functions[J].
Proceedings of the American Mathematical Society,
2006, 134(4): 1153-1160.

\bibitem{YSZ}
Yang Z H, Zheng S Z.
Complete monotonicity involving some ratios of gamma functions[J].
Journal of Inequalities and Applications,
2017, 2017(1): 255.

\bibitem{ZZH}
Yang Z H, Zheng S Z.
Complete monotonicity and inequalities involving gurland's ratios of gamma functions[J].
Mathematical Inequalities and Applications,
2019, 22(1): 97-109.

\bibitem{QLD}
Qi F, Lim D.
Monotonicity properties for a ratio of finite many gamma functions[J].
Advances in Difference Equations,
2020, 2020(1): 1-9.

\bibitem{QNL}
Qi F, Niu D W, Lim D, et al.
Some logarithmically completely monotonic functions and inequalities for
multinomial coefficients and multivariate beta functions[J].
Applicable Analysis and Discrete Mathematics,
2020, 14(2): 512-527.

\bibitem{BCK}
Berg C, Cetinkaya A, Karp D.
Completely monotonic ratios of basic and ordinary gamma functions[J].
Aequationes Mathematicae,
2021, 95(3): 569-588.

\bibitem{TZH}
Tian J F, Yang Z H.
Logarithmically complete monotonicity of ratios of $q$-gamma functions[J].
Journal of Mathematical Analysis and Applications,
2022, 508(1): 125868.

\bibitem{DA}
Kershaw D, Laforgia A.
Monotonicity results for the gamma function[J].
Atti della Accademia delle Scienze di Torino. Classe di Scienze Fisiche, Matematiche e Naturali,
1985, 119(3-4): 127-133.

\bibitem{BNF}
Qi F, Guo B N.
Some logarithmically completely monotonic functions related to the gamma function[J].
Journal of the Korean Mathematical Society,
2010, 47(6): 1283-1297.

\bibitem{CZZ}
Chu Y M, Zhang X M, Zhang Z H.
The geometric convexity of a function involving gamma function with applications[J].
Communications of the Korean Mathematical Society,
2010, 25(3): 373-383.

\bibitem{SLCY}
Qiu S L, Cai C Y.
Monotonicity and convexity properties of the gamma and psi functions[J].
Journal of Zhejiang Sci-Tech University,
2018, 39(1): 108-112.

\bibitem{FCP}
Qi F, Chen C P.
Monotonicity and convexity results for functions involving the gamma function[J].
International Journal of Applied Mathematical Sciences,
2004, 1: 27-36.

\bibitem{MQA}
An M Q.
Logarithmically complete monotonicity and logarithmically absolute monotonicity properties for the gamma function[J].
Communications in Mathematical Analysis,
2009, 6(2): 69-78.

\bibitem{MI}
Abramowitz M, Stegun I A.
Handbook of Mathematical Functions with Formulas, Graphs and Mathematical Tables[M].
10th ed. Dover, New York,
1972.

\bibitem{SM}
Qiu S L, Vuorinen M.
Special Functions in Geometric Function Theory[M].
Handbook of Complex Analysis: Geometric Function Theory,
2005, 2: 621-659.

\bibitem{GQ}
Guo S, Qi F, Srivastava H M.
A class of logarithmically completely monotonic functions related to the gamma function with applications[J].
Integral Transforms and Special Functions,
2012, 23(8): 557-566.

\bibitem{GMF}
Fichtenholz G M.
Differential- und Integralrechnung.$\uppercase\expandafter{\romannumeral2}$[M].
VEB Deutscher Verlag der Wissenschaften, Berlin,
1986.

\end{thebibliography}
\end{document}